\documentclass[reqno]{amsart}
\usepackage{amsthm,amsmath,amssymb}
\usepackage[pdftex]{graphicx}
\usepackage{braket}
\usepackage{color}

\title{A full-twist inequality for the $\nu^+$-invariant}
\author{Kouki Sato}
\date{}

\newtheorem{thm}{Theorem}[section]
\newtheorem{prop}[thm]{Proposition}
\newtheorem{lem}[thm]{Lemma}
\newtheorem{cor}[thm]{Corollary}
\newtheorem{claim}{Claim}

\theoremstyle{definition}

\newtheorem*{dfn}{Definition}

\newtheorem*{remark}{Remark}
\newtheorem*{acknowledge}{Acknowledgements}

\DeclareMathOperator{\punc}{punc}
\DeclareMathOperator{\Int}{Int}

\DeclareMathOperator{\lk}{lk}
\DeclareMathOperator{\Cont}{Cont}

\begin{document}
\maketitle

\begin{abstract}
Hom and Wu introduced a knot concordance invariant called $\nu^+$, which dominates many concordance invariants derived from Heegaard Floer homology. In this paper, we give a full-twist inequality for $\nu^+$. By using the inequality, we extend Wu's 
cabling formula for $\nu^+$ (which is proved only for particular positive cables) to all cables in the form of an inequality. 
In addition, we also discuss $\nu^+$-equivalence, which is an equivalence relation on the knot concordance group. We introduce a partial order on $\nu^+$-equivalence classes, and study its relationship to full-twists.
\end{abstract}

\section{Introduction}
\subsection{Full-twist inequality for $\nu^+$-invariant}

The $\nu^+$-invariant is a non-negative integer valued knot concordance invariant defined by Hom and Wu \cite{hom-wu}.
The $\nu^+$-invariant dominates many concordance invariants derived from Heegaard Floer homology, in terms of obstructions to sliceness.
In fact, Hom proves in \cite{hom} that for a given knot $K$ in $S^3$, 
if both $\nu^+(K)$ and $\nu^+(-K)$ are zero, then all invariants $\tau$, $\nu$,
$V_k$, $\gamma$, $\varepsilon$, $d(S^3_{p/q}(-),i)$ and $\Upsilon(t)$ agree with their values on the unknot.
(Here $-K$ denotes the orientation reversed mirror of $K$.)
Hence the $\nu^+$-invariant plays a special role among knot concordance invariants derived from Heegaard Floer homology.

In this paper, we give a full-twist inequality for the $\nu^+$-invariant.
To state the inequality, we first describe {\it full-twist operations}. 
Let $K$ be a knot in $S^3$ and $D$ a disk in $S^3$ which intersects $K$ in its interior.
By performing $(-1)$-surgery along $\partial D$, we obtain a new knot $J$ in $S^3$ from $K$.
Let $n = \lk(K, \partial D)$. Since reversing the orientation of $D$ does not affect the result, we may assume that $n\geq0$. Then we say that $K$ is deformed into $J$ by 
{\it a positive full-twist with $n$-linking}, and call such an operation a {\it full-twist operation}. The main theorem of this paper is stated as follows. 
(Here $\#$ denotes connected sum.)
\begin{thm}
\label{thm1}
Suppose that a knot $K$ is deformed into a knot $J$ by a positive full-twist with $n$-linking. If $n=0$, then $\nu^+(J\#(-K))=0$. Otherwise, we have
$$
\frac{(n-1)(n-2)}{2} \leq \nu^+(J\#(-K)) \leq \frac{n(n-1)}{2}.
$$
\end{thm}

\begin{remark}
For any coprime $p,q>0$, let $T_{p,q}$ denote the $(p,q)$-torus knot.
Then we note that $\nu^+(T_{p,q})= (p-1)(q-1)/2$ \cite{hom-wu, tau}, and hence the inequality in Theorem \ref{thm1} implies 
$$
\nu^+(T_{n,n-1}\#K\#(-K)) \leq \nu^+(J\#(-K)) \leq \nu^+(T_{n,n+1}\#K\#(-K)).
$$
Since both $T_{n,n-1}\#K$ and $T_{n,n+1}\#K$ are obtained from $K$ by a positive full-twist with $n$-linking, the inequalities are best possible.
\end{remark}

Here we note that Theorem \ref{thm1} gives an inequality for $J\# (-K)$ rather than $J$ and $K$. However, by subadditivity of $\nu^+$ \cite{bodnar-celoria-golla}, we also have the following result for $J$ and $K$.
\begin{thm}
\label{thm2}
Suppose that $K$ is deformed into $J$ by a positive full-twist with $n$-linking.
If $n=0$, then $\nu^+(J) \leq \nu^+ (K)$. Otherwise, we have
$$
\frac{(n-1)(n-2)}{2} - \nu^+(-K) \leq \nu^+(J) \leq  \frac{n(n-1)}{2} + \nu^+(K).
$$
\end{thm}
Furthermore,
we can use Theorem \ref{thm2} to obtain the following lower bound for the $\nu^+$-invariant of all cable knots (including negative cables).

\begin{thm}
\label{thm3}
For any knot $K$ and coprime integers $p,q$ with $p>0$, we have
$$
\nu^+(K_{p,q}) \geq p \nu^+(K) + \frac{(p-1)(q-1)}{2},
$$
where $K_{p,q}$ denotes the the $(p,q)$-cable of $K$.
satisfies
\end{thm}

Note that Wu proves in \cite{wu} that the equality holds in the case
where $p,q>0$ and $q\geq (2\nu^+(K)-1)p-1$. Hence Theorem \ref{thm3} partially extends his result
to arbitrary cables. Furthermore, Theorem \ref{thm3} also enables us 
to extend Wu's 4-ball genus bound for particular positive cable knots to all positive cable knots.

\begin{cor}
\label{cor2}
If $\nu^+(K) = g_4(K)$, then for any coprime $p,q>0$, we have
$$
\nu^+(K_{p,q}) = g_4(K_{p,q}) = p g_4(K) + \frac{(p-1)(q-1)}{2}.
$$
\end{cor}

As an application, for instance, we can determine the 4-ball genus for all positive cables of
the knot $T_{2,5}\#T_{2,3}\#T_{2,3}\#(-(T_{2,3})_{2,5})$. This example is used in \cite{hom-wu} to show that $\nu^+ \neq \tau$. Remark that the $\tau$-invariant cannot determine the 4-ball genus for any positive cable of the knot.
Also note that this generalizes \cite[Proposition 3.5]{hom-wu} and Wu's result in the introduction of \cite{wu}. 


\subsection{A partial order on $\nu^+$-equivalence classes}

Let $\mathcal{C}$ denote the knot concordance group. For two elements 
$x,y \in \mathcal{C}$, we say that $x$ is {\it $\nu^+$-equivalent} to $y$
if the equalities $\nu^+(x-y)=0$ and $\nu^+(y-x)=0$ hold.
In \cite{hom}, Hom proves that $\nu^+$-equivalence is an equivalence relation and the quotient has a group structure as a quotient group of $\mathcal{C}$
(we denote it by $\mathcal{C}_{\nu^+}$). Furthermore, 
it follows from \cite[Theorem 1]{hom} that
the invariants $\tau, \Upsilon, V_k, \nu^+$ and $d(S^3_{p/q}(\cdot),i)$ are invariant under not only knot concordance but also $\nu^+$-equivalence.
In particular, we can regard these invariants as maps
$$
\begin{array}{rll}
\tau: &\mathcal{C}_{\nu^+} \to \mathbb{Z},&\ \\
\Upsilon:&\mathcal{C}_{\nu^+} \to \Cont([0,2]),&\ \\
V_k: &\mathcal{C}_{\nu^+} \to \mathbb{Z}_{\geq 0},&\ \\
\nu^+: &\mathcal{C}_{\nu^+} \to \mathbb{Z}_{\geq 0}, &\ \text{ and}\\
d(S^3_{p/q}(\cdot),i): &\mathcal{C}_{\nu^+} \to \mathbb{Q}&\ 
\end{array}
$$
respectively. (Here, $\Cont([0,2])$ denotes the set of continuous functions from the closed interval $[0,2]$ to $\mathbb{R}$.)

The second aim of this paper is to introduce a partial order on $\mathcal{C}_{\nu^+}$
and discuss its relationship to full-twists.
The precise definition of the partial order is as follows.
\begin{dfn}
For two elements $x,y \in \mathcal{C}_{\nu^+}$, we write $x \leq y$
if  $\nu^+(x-y)=0$.
\end{dfn}

Note that the equality in the above definition is one of the equalities in the definition of 
$\nu^+$-equivalence, and so this partial order seems to be very natural.
In fact, we can prove the following proposition.

\begin{prop}
\label{prop partial}
The relation $\leq$ is a partial order on $\mathcal{C}_{\nu^+}$ with the following properties;
\begin{enumerate}
\item For elements $x, y, z \in \mathcal{C}_{\nu^+}$, if $ x \leq y$, then 
$x+z \leq y+z$.
\item For elements $x, y \in \mathcal{C}_{\nu^+}$, if $ x \leq y$, then 
$-y \leq -x$.
\item For coprime integers $p,q>0$, $k \in \mathbb{Z}_{\geq 0}$ and $0 \leq i \leq p-1$,
all of $\tau$, $\nu$, $-\Upsilon$, $V_k$, $\nu^+$
and $-d(S^3_{p/q}(\cdot),i)$ preserve the partial order.
\end{enumerate}
\end{prop}

Here the third assertion in Proposition \ref{prop partial} implies that there are many algebraic obstructions to one element of $\mathcal{C}_{\nu^+}$ being less than another.
 On the other hand, the following theorem establishes similar obstructions in terms of geometric deformations.

\begin{thm}
\label{thm partial}
Suppose that $K$ is deformed into $J$ by a positive full-twist with $n$-linking.
\begin{enumerate}
\item If $n= 0 \text{ or } 1$, then $[J]_{\nu^+} \leq [K]_{\nu^+}$.
\item If $n \geq 3$, then $[J]_{\nu^+} \nleq [K]_{\nu^+}$. In particular, if the geometric intersection number between $K$ and $D$ is equal to $n$, then $[J]_{\nu^+} > [K]_{\nu^+}$.
\end{enumerate}
Here $[K]_{\nu^+}$ denotes the $\nu^+$-equivalence class of a knot $K$, and 
the symbol $>$ means $x \geq y$ and $x \neq y$ for elements $x,y \in \mathcal{C}_{\nu^+}$.
\end{thm}

In the above theorem, we can see that only the case of $n=2$ tells us nothing about the partial order. This follows from the fact that Theorem \ref{thm1} gives
$0 \leq \nu^+(x-y) \leq 1$ for $n=2$ and hence we can show neither 
$\nu^+(x-y) = 0$ nor $\nu^+(x-y)\neq 0$. 

We also mention the relationship between our partial order and satellite knots.
Let $P$ be a knot in a standard solid torus $V \subset S^3$ with the longitude $l$, and $K$ a knot in $S^3$. For $n \in \mathbb{Z}$, Let $e_n: V \to S^3$ be
an embedding so that $e(V)$ is a tubular neighborhood of $K$ and $\lk(K,e_n(l))=n$. 
Then we call $e_n(P)$ the {\it $n$-twisted satellite knot of $K$ with pattern $P$},
and denote it by $P(K,n)$.
Furthermore, if $P$ represents $m$ times generators of $H_1(V;\mathbb{Z})$ for $m\geq0$,
then we denote $w(P):=m$.
It is proved in \cite[Theorem B]{kim-park} that the map $[K]_{\nu^+} \mapsto [P(K,n)]_{\nu^+}$ for any pattern $P$ with $w(P) \neq 0$.
We extend their theorem to all satellite knots, and show that those maps preserve our partial order.

\begin{prop}\label{prop satellite}
For any pattern $P$ and $n \in \Bbb{Z}$, the map $P_n: \mathcal{C}_{\nu^+} \to \mathcal{C}_{\nu^+}$ defined by
$P_n([K]_{\nu^+}) := [P(K,n)]_{\nu^+}$ is well-defined and preserve the partial order $\leq$.
\end{prop}

By Proposition \ref{prop satellite}, we obtain infinitely many order-preserving maps on
$\mathcal{C}_{\nu^+}$ which have geometric meaning. 
Now it is an interesting problem to compare these satellite maps.
Theorem \ref{thm partial} tells us the relationship among the maps $\{P_n\}_{n\in \Bbb{Z}}$
for some particular patterns.
\begin{cor}
\label{cor satellite}
Let $P$ be a pattern.
\begin{enumerate}
\item If $w(P)= 0 \text{ or } 1$, then the inequality $P_m (x) \geq P_n(x)$ holds 
for any integers $m < n$ and $x \in \mathcal{C}_{\nu^+}$.
\item If the geometric intersection number between $P$ and the meridian disk of $V$ is equal to $w(P)$ and $w(P)\geq3$, then $P_m (x) < P_n(x)$ for any $m < n$ and $x \in \mathcal{C}_{\nu^+}$.
\end{enumerate}
\end{cor}


\subsection{The idea of proofs: study of slice knots in $\mathbb{C}P^2$}

In this subsection, we explain the idea of our proof of Theorem \ref{thm1}.
We start from an interpretation of full-twist operations in 4-dimensional topology.

When a knot $K$ is deformed into a knot $J$ 
by a positive full-twist with $n$-linking,
we can see that the knot $J\#(-K)$ bounds a disk $D$ in $\punc \overline{\mathbb{C}P^2}$.
(Here, for a closed 4-manifold $X$, $\punc X$ denotes $X$ with an open 4-ball deleted.) 
In particular, the disk $D$ represents $n$ times a generator of 
$H_2(\punc \overline{\mathbb{C}P^2}, \partial (\punc \overline{\mathbb{C}P^2}); \mathbb{Z})
\cong \mathbb{Z}$.
In this situation, we consider Ni-Wu's $V_k$-sequence \cite{ni-wu}. 

\begin{prop}
\label{prop cp2}
Suppose that a knot $K$ in $S^3$ bounds a disk $D$ in $\punc \overline{\mathbb{C}P^2}$
such that 
$[D,\partial D] = n \gamma \in H_2(\punc \overline{\mathbb{C}P^2}, \partial (\punc \overline{\mathbb{C}P^2}); \mathbb{Z})$ for a generator $\gamma$ and some integer
$n\geq0$.
\begin{enumerate}
\item If $n=0$, then $V_0(K)=0$.

\item If $n$ is odd, then for any $0 \leq j \leq \frac{n-1}{2}$, we have
$$
V_{nj}(K) = \frac{1}{2}\left(\frac{n-1}{2}-j \right) \left(\frac{n-1}{2}-j +1 \right).
$$
\item If $n$ is even and $n>0$, then for any $0 \leq j \leq \frac{n}{2}-1$, we have
$$
V_{\frac{n}{2} + nj}(K) = \frac{1}{2} \left( \frac{n}{2}-j \right) \left( \frac{n}{2}-j-1 \right).
$$
\end{enumerate}
\end{prop}

In fact, Theorem \ref{thm1} is an immediate consequence of Proposition \ref{prop cp2},
and this proposition is much stronger than Theorem \ref{thm1} when one wants obstructions to the existence of a positive full-twist between two knots. 
We will study the strength of Proposition \ref{prop cp2} in future work.
\begin{acknowledge}
The author was supported by JSPS KAKENHI Grant Number 15J10597.
The author would like to thank his supervisor, Tam\'{a}s K\'{a}lm\'{a}n
for his encouragement.
The author also would like to thank 
Jennifer Hom for her useful comments.
The proof of Proposition \ref{prop satellite} was derived in discussion with Hironobu Naoe.
\end{acknowledge}


\section{Preliminaries}

In this section, 
we recall some invariants derived from Heegaard Floer homology. 

\subsection{Correction terms}
Ozsv\'{a}th and Szab\'{o} \cite{ozsvath-szabo} introduced a $\mathbb{Q}$-valued invariant
$d$ (called the {\it correction term})
for rational homology 3-spheres endowed with a Spin$^c$ structure.
It is proved that the correction term is invariant under Spin$^c$ rational homology cobordism. In particular, the following proposition holds.

\begin{prop}
\label{prop2.1}
If a rational homology 3-sphere $Y$ bounds a rational homology 4-ball $W$,
then for any Spin$^c$ structure $\frak{s}$ over $W$, we have
$$
d(Y,\frak{s}|_{Y})=0,
$$
where $\frak{s}|_{Y}$ denotes the restriction of $\frak{s}$ to $Y$.
\end{prop}

\subsection{$V_k$-sequence and $\nu^+$-invariant}
The {\it $V_k$-sequence } is a family of $\mathbb{Z}_{\geq 0}$-valued knot concordance invariants $\{V_k(K)\}_{k \geq 0}$ defined by Ni and Wu \cite{ni-wu}. 
In particular, $\nu^+(K) := \min \{ k \geq 0 \mid V_k(K) =0 \}$ is known as the {\it $\nu^+$-invariant} \cite{hom-wu}.

In \cite{ni-wu}, Ni and Wu prove that the set $\{V_k(K)\}_{k \geq 0}$ determines all correction terms of the $p/q$-surgery along $K$ for any coprime $p,q>0$. 
Let $S^3_{p/q}(K)$ denote the $p/q$-surgery along $K$.
Note that there is a canonical identification between the set of Spin$^c$ structures over $S^3_{p/q}(K)$ and $\{ i \mid 0 \leq i \leq p-1 \}$.
(This identification can be made explicit by the procedure in \cite[Section 4, Section 7]{ozsvath-szabo4}.)

\begin{prop}[\text{\cite[Proposition 1.6]{ni-wu}}]
\label{prop2.2}
Suppose $p,q>0$, and fix $0 \leq i < p$. Then
$$
d(S^3_{p/q}(K),i) = d(S^3_{p/q}(O),i) - 2 
V_{\min \left\{ \lfloor \frac{i}{q} \rfloor, \lfloor \frac{p+q -1-i}{q} \rfloor \right\} }(K),
$$
where $O$ denotes the unknot and $\lfloor \cdot \rfloor$ is the floor function.
\end{prop}

Here we  modify the right hand side of the equality by using the fact that $\{V_k(K)\}_{k \geq 0}$ satisfy the inequality $V_{k+1}(K) \leq V_{k}(K)$ for each $k \geq 0$.
In particular, for integer surgeries, we have the following formula.
\begin{cor}
\label{cor2.3}
For any $p>0$ and $0 \leq i < p$, we have
$$
d(S^3_{p}(K),i) = d(S^3_{p}(O),i) - 2 
V_{\min\{ i,p-i\} }(K).
$$
\end{cor}


\section{Proof of Proposition \ref{prop cp2}}

In this section, we prove Proposition \ref{prop cp2}.
We start from the following lemma.

\begin{lem}
\label{lem3.1}
Suppose that a knot $K$ satisfies the assumption of Proposition \ref{prop cp2}
and $n>0$.
then $S^3_{n^2}(K)$ bounds a rational homology 4-ball.
\end{lem}

\begin{proof}
Consider a standard handle decomposition 
$\overline{\mathbb{C}P^2} = h^0 \cup h^2 \cup h^4$,
where $h^i$ denotes its unique $i$-handle for $i \in \{0,2,4\}$.
Think of $\punc \overline{\mathbb{C}P^2}$ as 
$\overline{\mathbb{C}P^2} \setminus \Int h^0$, and suppose that $K$ lies in 
$\partial(\overline{\mathbb{C}P^2} \setminus \Int h^0 ) =  -\partial h^0$.
Then we can think of a tubular neighborhood of $D$ (with orientation reversed) as an $(-n^2)$-framed 2-handle $\widetilde{h^2}$ attached to $h^0$ along $-K$ (see Figure \ref{X}).
Let $X := h^0 \cup \widetilde{h^2}$. Then $X$ is a codimension 0 sub-manifold of $\overline{\mathbb{C}P^2}$
which satisfies;
\begin{enumerate}
\item $H_*(X;\mathbb{Z}) \cong H_*(S^2;\mathbb{Z})$,
\item the induced map $i_*:H_2(X;\mathbb{Z}) \to H_2(\overline{\mathbb{C}P^2};\mathbb{Z})$
 from the inclusion is non-trivial, and
\item  $\partial X = S^3_{-n^2}(-K)$.
\end{enumerate}

These imply that for the exterior $W := \overline{\mathbb{C}P^2} \setminus \Int X$,
we have $\partial W = -S^3_{-n^2}(-K) \cong S^3_{n^2}(K)$ and 
$H_*(W;\mathbb{Q}) \cong H_*(B^4;\mathbb{Q})$.
\end{proof}

\begin{figure}[htbp]
\begin{center}
\includegraphics[ scale = 0.8]{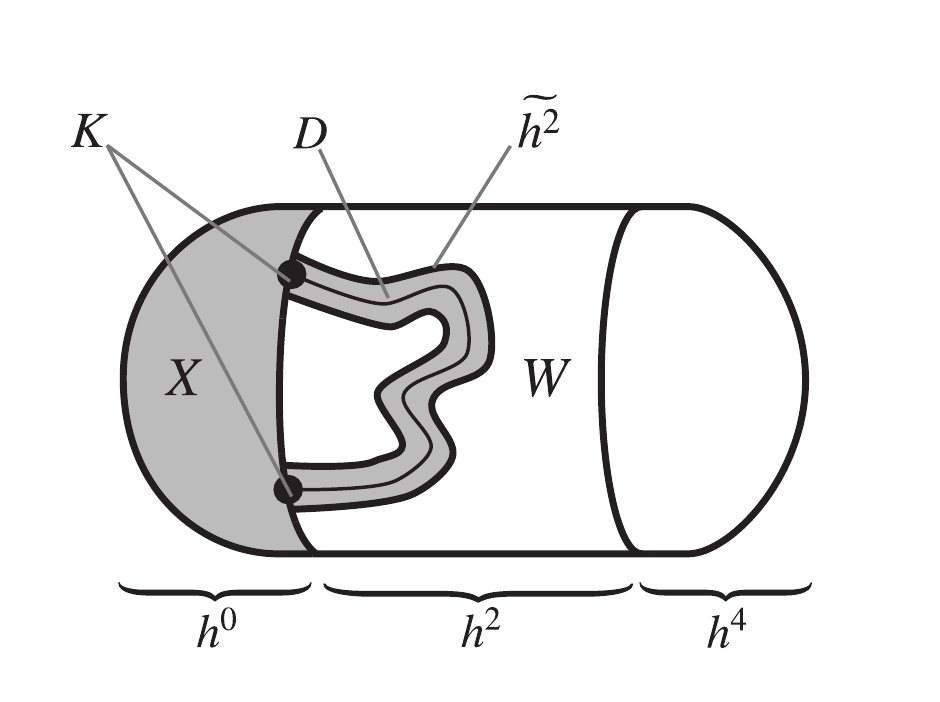}
\vspace{-5mm}
\caption{\label{X}}
\end{center}
\end{figure}

In light of Proposition \ref{prop2.2},
Lemma \ref{lem3.1} shows that the correction term of some Spin$^c$ structures over 
$S^3_{n^2}(K)$ are zero. In fact, it follows from \cite[Proposition 4.1]{owens-strle}.
that the number of such Spin$^c$ structures is at least $n$. 
\begin{lem}
\label{lem3.2}
If $S^3_{n^2}(K)$ bounds a rational homology 4-ball, 
then there exists a subset $S \subset \{ i \mid 0\leq i<n^2 \}$
such that $|S|=n$ and for any element $i \in S$, we have $d(S^3_{n^2}(K),i)=0$.
\end{lem}
We will determine the subset $S$.
Let $f_p: \{i \mid 0 \leq i < p\} \to \mathbb{Q}$ be a map
defined by
$$
f_p(i)= \frac{-p+(p-2i)^2}{4p}.
$$
Actually, $f_p(i)$ is the value of $d(S^3_{p}(O),i)$.
By Corollary \ref{cor2.3}, we can see that if an element $i \in \{0 \leq i < n^2 \}$ belongs to
$S$, then $f_{n^2}(i)$ must be an even integer.
So we next observe when $f_{n^2}(i)$ is an even integer.
\begin{lem}
\label{lem3.3}
If $n$ is odd, then 
$$
f_{n^2}(i) \in 2\mathbb{Z} \Leftrightarrow i \in n\mathbb{Z}.
$$
If $n$ is even, then 
$$
f_{n^2}(i) \in 2\mathbb{Z} \Leftrightarrow i-\frac{n}{2} \in n\mathbb{Z}.
$$
\end{lem}

\def\proofname{Proof}

\begin{proof}
We first assume that $n$ is odd.
It is easy to check that the following equalities hold;
\begin{eqnarray}
f_{n^2}(i)&=&\left( \frac{n-1}{2} - \frac{i}{n} \right) \left( \frac{n-1}{2} - \frac{i}{n} +1 \right) \\
\ &=& \frac{(n+1)(n-1)}{4} -i + \left( \frac{i}{n} \right)^2.
\end{eqnarray}
Then the equality (1) implies that $f_{n^2}(i) \in 2\mathbb{Z}$ if $i \in n \mathbb{Z}$,
while the equality (2) implies that $f_{n^2}(i) \in \mathbb{Z}$ only if $(i/n)^2 \in \mathbb{Z}$.
Since $(i/n)^2 \in \mathbb{Z}$ if and only if $i \in n \mathbb{Z}$,
this proves Lemma \ref{lem3.3} for odd $n$ .

Next we assume that $n$ is even.
We can see that the following equalities also hold;
\begin{eqnarray}
f_{n^2}(i)&=&
\left( \frac{n}{2} - \frac{i-\frac{n}{2}}{n} \right) 
\left( \frac{n}{2} - \frac{i-\frac{n}{2}}{n} -1 \right) \\
\ &=& \left(\frac{n}{2} \right)^2 -i + \left(\frac{i-\frac{n}{2}}{n} \right) 
+ \left( \frac{i-\frac{n}{2}}{n} \right)^2.
\end{eqnarray}
Let $x = (i-n/2)/n$. Then the equality (3) implies that $f_{n^2}(i) \in 2\mathbb{Z}$ if $x \in \mathbb{Z}$,
while the equality (4) implies that $f_{n^2}(i) \in \mathbb{Z}$ only if $x + x^2 \in \mathbb{Z}$.
Furthermore,  it is not hard to verify that 
$x+x^2 \in \mathbb{Z}$ if and only if $x \in \mathbb{Z}$.
This completes the proof.
\end{proof}

Now we have 
$$
S \subset \{  0 \leq i < n^2 \mid f_{n^2}(i) \in 2\mathbb{Z} \}
=
\left\{
\begin{array}{ll}
\{ 0 \leq i < n^2 \mid i \in n\mathbb{Z} \}& (n \text{: odd})\\
\ & \ \\
\{0 \leq i < n^2 \mid i - n/2 \in n\mathbb{Z} \}& (n \text{: even})
\end{array}
\right.
.
$$
However, for any $n$, the order of the rightmost set is $n$.
This implies
$$
S =
\left\{
\begin{array}{ll}
\{ 0 \leq i < n^2 \mid i \in n\mathbb{Z} \}& (n \text{: odd})\\
\ & \ \\
\{ 0 \leq i < n^2 \mid i - n/2 \in n\mathbb{Z} \}& (n \text{: even})
\end{array}
\right.
.
$$
Hence, combining the above three lemmas, we obtain the following lemma.
\begin{lem}
\label{lem3.4}
Suppose that a knot $K$ satisfies the assumption of Proposition \ref{prop cp2}
and $n>0$.
If $n$ is odd, then we have 
$$
d(S^3_{n^2}(K),nj)=0 \ (0 \leq ^{\forall}j \leq n-1).
$$
If $n$ is even, then  we have
$$
d(S^3_{n^2}(K),\frac{n}{2} + nj)=0 \ (0 \leq ^{\forall}j \leq n-1).
$$
\end{lem}
Now let us prove Proposition \ref{prop cp2}.
\def\proofname{Proof of Proposition \ref{prop cp2}}

\begin{proof}
We first suppose that $n=0$. Then, the knot $K$ bounds a null-homologous disk in $\punc \overline{\mathbb{C}P^2}$, whose intersection form is negative definite. Such a knot is studied in author's paper \cite{sato}, and it is proved that $V_0(K)=0$. Hence the assertion (1) of Proposition \ref{prop cp2} holds.

To prove the other two assertions, we use the equality
\begin{equation}
d(S^3_{n^2}(K),i) = f_{n^2}(i) -2V_{i}(K) \ (0 \leq ^{\forall} i \leq n^2/2),
\end{equation}
which is obtained by Corollary \ref{cor2.3}.
We first consider the case where $n>0$ is odd and $0 \leq j \leq (n-1)/2$. Then the inequalities $0 \leq nj \leq n(n-1)/2 < n^2/2$
hold, and hence the equality (5) and Lemma \ref{lem3.4} give 
$$
0 = f_{n^2}(nj) -2V_{nj}(K).
$$
This equality and the equality (1) prove the assertion (2) of Proposition \ref{prop cp2}. 

We next consider the case where $n>0$ is even and $0 \leq j \leq n/2 -1$. Then the inequalities
$0 \leq n/2 + nj \leq n^2/2 -n/2 < n^2/2$ hold, and hence the equality (5) and 
Lemma~\ref{lem3.4}
give
$$
0 = f_{n^2}(\frac{n}{2} + nj) -2V_{\frac{n}{2} +nj}(K).
$$
This equality and the equality (3) prove the assertion (3) of Proposition \ref{prop cp2}.
\end{proof}

\begin{remark}
We can prove Lemma \ref{lem3.4}
by taking suitable Spin$^c$ structures over $\overline{\mathbb{C}P^2}$ 
and restricting them to $W$.
This alternate proof seems to be natural rather than the original proof,
while the original proof also shows that any other correction term of 
$S^3_{n^2}(K)$ is not lying even in $\mathbb{Z}$. 
\end{remark}


\section{Proof of full-twist inequalities}

\def\proofname{Proof of Theorem \ref{thm1}}

In this section, we prove Theorem \ref{thm1} and Theorem \ref{thm2}.
We first prove Theorem \ref{thm1}.

\begin{proof}
By the assumption of Theorem \ref{thm1} and the definition of a full-twist operation,
there exists a disk $D$ in $S^3$ which intersects $K$ in its interior, and after $(-1)$-surgery along $\partial D$, we obtain $J$ from $K$.
Consider a standard handle decomposition 
$\overline{\mathbb{C}P^2} = h^0 \cup h^2 \cup h^4$.
Then we have an annulus $A$ properly embedded in 
$\overline{\mathbb{C}P^2} \setminus (\Int h^0 \cup \Int h^4)$
such that 
$(\overline{\mathbb{C}P^2} \setminus (\Int h^0 \cup \Int h^4),A)$ is a pair of cobordism 
from $(\partial h^0, K)$ to 
$(\partial(\overline{\mathbb{C}P^2} \setminus \Int h^4), J)$ (see Figure \ref{A}).
Furthermore, the annulus $A$ induces an annulus $A'$ in 
$\overline{\mathbb{C}P^2} \setminus (\Int h^0 \cup \Int h^4)$, which connects $(\partial h^0 ,K\# (-K))$ to 
$(\partial(\overline{\mathbb{C}P^2} \setminus \Int h^4), J \# (-K))$ (Figure \ref{A'}).
Since $K \# (-K)$ bounds a disk $D'$ in $h^0 \cong B^4$,
by gluing $D'$ with $A'$ along $K\#(-K)$, we have a disk $D''$ in 
$\partial(\overline{\mathbb{C}P^2} \setminus \Int h^4$) with boundary $J \# (-K)$.
Here we note that the disk $D''$ represents  
$n\gamma \in H_2(\overline{\mathbb{C}P^2} \setminus h^4, 
\partial (\overline{\mathbb{C}P^2} \setminus h^4);\mathbb{Z})$
for a generator $\gamma$ and $n= \lk(K, \partial D) \geq 0$.

\begin{figure}[tbp]
\begin{minipage}[]{0.4\hsize}
\hspace{-8mm}
\includegraphics[scale = 0.65]{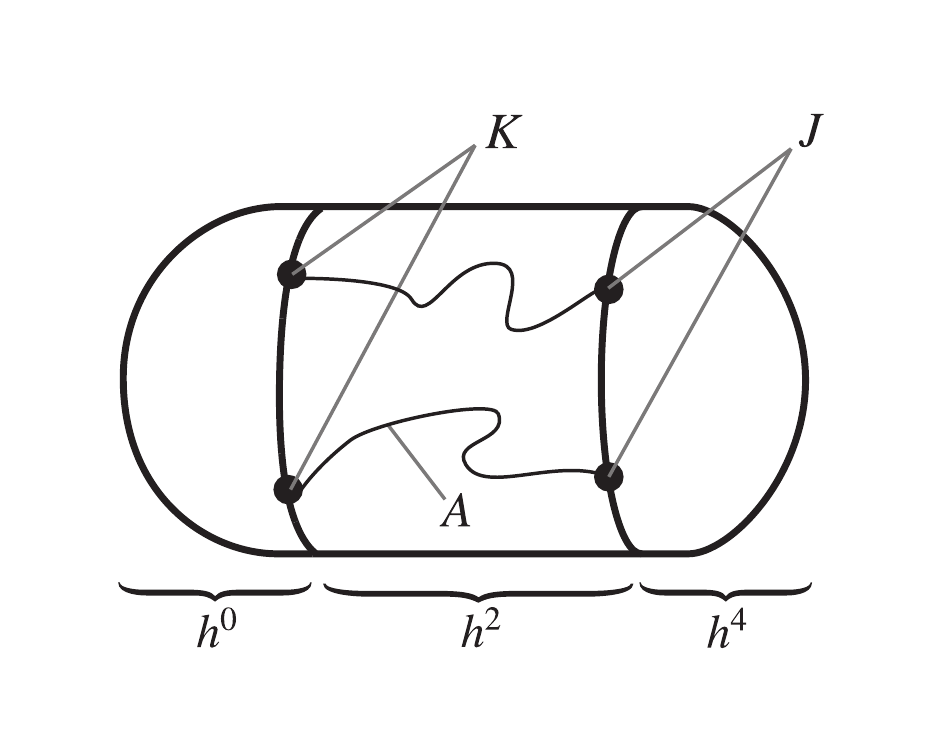}
\vspace{-10mm}
\caption{\label{A}}
 \end{minipage}
\hspace{4mm}
\begin{minipage}[]{0.4\hsize}
\hspace{-8mm}
\includegraphics[scale = 0.65]{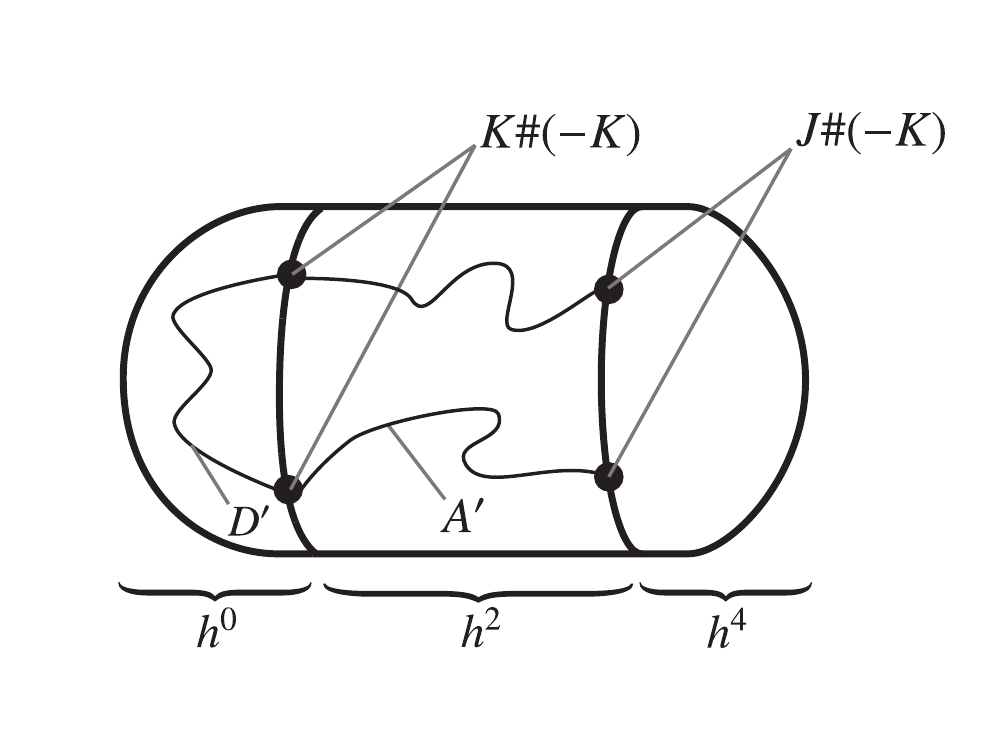}
\vspace{-10mm}
\caption{\label{A'}}
 \end{minipage}
\end{figure}

Now we can apply Proposition \ref{prop cp2} to the pair $(D'',J \# (-K))$.
In particular, if $n=0$, then $V_0(J\#(-K)) = 0$, and we have
$$
\nu^+(J\#(-K)) = \min\{k \geq 0 \mid V_k(J\#(-K)) = 0\} = 0.$$ 
We consider the case where $n$ is odd. 
Then for any $0 \leq j \leq \frac{n-1}{2}$, we have
$$
V_{nj}(J\#(-K)) = \frac{1}{2}\left(\frac{n-1}{2}-j \right) \left(\frac{n-1}{2}-j +1 \right).
$$
In particular, the equality for $j = (n-1)/2$ gives $V_{\frac{n(n-1)}{2}}(J \#(-K)) = 0$.
Moreover, if $n>1$, then the equality for $j = (n-3)/2$ gives
$V_{\frac{n(n-3)}{2}}(J \#(-K)) = 1$.
These imply that for $n>1$, we have
$$
\frac{(n-1)(n-2)}{2} = \frac{n(n-3)}{2}+1 \leq \nu^+(J \#(-K))  \leq \frac{n(n-1)}{2}.
$$
On the other hand, if $n=1$, then the equality $V_{\frac{n(n-1)}{2}}(J \#(-K)) = 0$ directly shows
$$
\frac{(n-1)(n-2)}{2} = 0 \leq \nu^+(J\#(-K)) \leq \frac{n(n-1)}{2}.
$$

We next consider the case where $n$ is even and $n>0$.
In this case, Proposition \ref{prop cp2} gives 
$$
V_{\frac{n}{2} + nj}(K) = \frac{1}{2} \left( \frac{n}{2}-j \right) \left( \frac{n}{2}-j-1 \right)
$$
for any $0 \leq j \leq \frac{n}{2}-1$.
In particular, the equality for $j = n/2-1$ gives $V_{\frac{n^2}{2} - \frac{n}{2}}(J \#(-K)) = 0$,
Moreover, if $n>2$, then the equality for $j = n/2 -2$ gives
$V_{\frac{n^2}{2} - \frac{3n}{2}}(J \#(-K)) = 1$.
These imply that for $n>2$, we have
$$
\frac{(n-1)(n-2)}{2} = \frac{n^2}{2} - \frac{3n}{2}+1 
\leq \nu^+(J \#(-K))  \leq \frac{n^2}{2} - \frac{n}{2} = \frac{n(n-1)}{2}.
$$
On the other hand, if $n=2$, then the equality $V_{\frac{n^2}{2} - \frac{n}{2}}(J \#(-K)) = 0$ 
directly shows
$$
\frac{(n-1)(n-2)}{2} = 0 \leq \nu^+(J\#(-K)) \leq \frac{n^2}{2} - \frac{n}{2} = \frac{n(n-1)}{2}.
$$
This completes the proof
\end{proof}

We next prove Theorem \ref{thm2}. To prove it, we use the following theorem.

\begin{thm}[\text{\cite[Theorem 1.4]{bodnar-celoria-golla}}]
\label{thm BCG}
For any two elements $x,y \in \mathcal{C}_{\nu^+}$, we have
$$
\nu^+(x+y) \leq \nu^+(x) + \nu^+(y).
$$
\end{thm}

\def\proofname{Proof of Theorem \ref{thm2}}

\begin{proof}
Under the assumption of Theorem \ref{thm2} for $n>0$, it follows from Theorem \ref{thm1}
and Theorem \ref{thm BCG} that
$$
\frac{(n-1)(n-2)}{2} \leq \nu^+(J\#(-K)) \leq \nu^+(J) + \nu^+(-K)
$$
and
$$
\nu^+(J)=\nu^+(J\#(-K)\#K) \leq \nu^+(J\#(-K)) + \nu^+(K) \leq 
\frac{n(n-1)}{2} + \nu^+(K).
$$
On the other hand, if $n=0$, we have
$$
\nu^+(J)=\nu^+(J\#(-K)\#K) \leq \nu^+(J\#(-K)) + \nu^+(K) = \nu^+(K).
$$
These complete the proof.
\end{proof}


\section{4-ball genus bound for positive cable knots}

In this section, we prove Theorem \ref{thm3} and Corollary \ref{cor2}. 
Before proving Theorem \ref{thm3}, we note that it has been proved in \cite{wu} that the equality in Theorem \ref{thm3} holds for sufficiently large $q$ relative to $p$.

\begin{thm}[\text{\cite[Theorem 1.1]{wu}}]
\label{thm wu}
Let $K$ be a knot and
$p,q>0$ coprime integers with $q \geq (2\nu^+(K)-1)p -1$.
Then we have
$$
\nu^+(K_{p,q})=p\nu^+(K)+ \frac{(p-1)(q-1)}{2}.
$$
\end{thm}

In our proof, we deform a given $(p,q)$-cable (with $p>0$) into $(p,np+q)$-cable by $n$ times positive full twists for sufficiently large $n$, and apply Wu's theorem to the $(p,np+q)$-cable.

\def\proofname{Proof of Theorem \ref{thm3}}

\begin{proof}
Without loss of generality, we can assume that $p>0$.
We can easily see that any cable knot $K_{p,q}$ ($p>0$) is deformed into $K_{p,p+q}$ by  a positive full-twist with $p$-linking.
Hence, for any $n>0$, 
by taking $n$ times of such positive full-twists and applying Theorem \ref{thm2},
we have
\begin{equation}
\nu^+(K_{p,q}) \geq \nu^+(K_{p,np+q}) -  \frac{np(p-1)}{2}. 
\end{equation}
Let $n$ be a positive integer satisfying $n \geq 2 \nu^+(K) - 1 + |\frac{q}{p}|$. Then $np+q  \geq (2\nu^+(K)-1)p -1$,
and Theorem \ref{thm wu} gives
\begin{eqnarray}
\nu^+(K_{p,np+q}) &= &p\nu^+(K)+ \frac{(p-1)(np+q-1)}{2} \nonumber \\ 
\ &=& p\nu^+(K)+ \frac{(p-1)(q-1)}{2} + \frac{np(p-1)}{2}.
\end{eqnarray}
By combining the inequality (6) and the equality (7), we obtain
the desired inequality.
\end{proof}

Next we prove Corollary \ref{cor2}.

\def\proofname{Proof of Corollary \ref{cor2}}

\begin{proof}
Since one can construct a slice surface for $K_{p,q}$ from $p$ parallel copies of a 
slice surface for $K$ together with $(p-1)q$ half-twisted bands,
we have the inequality
$$
g_4(K_{p,q}) \leq pg_4(K) + \frac{(p-1)(q-1)}{2}.
$$
On the other hand, the assumption $\nu^+(K) = g_4(K)$ and Theorem \ref{thm3}
imply
$$
g_4(K_{p,q}) \geq \nu^+(K_{p,q}) \geq pg_4(K) + \frac{(p-1)(q-1)}{2}.
$$
These complete the proof.
\end{proof}


\section{A partial order on $\mathcal{C}_{\nu^+}$}

In this section, we prove Proposition \ref{prop partial}  and Theorem \ref{thm partial}.
We first prove Proposition \ref{prop partial}.

We decompose Proposition \ref{prop partial} into the following three lemmas.

\begin{lem}
\label{lem partial1}
The relation $\leq$ is a partial order on $\mathcal{C}_{\nu^+}$.
\end{lem}

\def\proofname{Proof}

\begin{proof}
To prove Lemma \ref{lem partial1}, it is sufficient to show that the followings hold;
for all $x,y,z \in \mathcal{C}_{\nu^+}$, we have
\begin{enumerate}
\item $x \leq x$,
\item if $x \leq y$ and $y \leq z$, then 
$x \leq z$, and
\item if $x \leq y$ and $y \leq x$, then 
$x = y$.
\end{enumerate}
Since $\nu^+(x-x) = \nu^+(0) = 0$, the condition (1) holds.
We next prove the condition (3). Since the assumptions $x \leq y$ and $y \leq x$
imply that $\nu^+(x-y)=0$ and $\nu^+(y-x)=0$,
representatives of $x$ and $y$ are $\nu^+$-equivalent, and we have $x=y$.

To prove the condition (2), we use Theorem \ref{thm BCG} again.
The assumptions $x \leq y$, $y \leq z$ and Theorem \ref{thm BCG}
imply that
$$
\nu^+(x-z) = \nu^+((x-y) + (y-z)) \leq \nu^+(x-y) + \nu^+(y-z)=0.
$$
This completes the proof.
\end{proof}

\begin{lem}
\label{lem partial2}
For elements $x, y, z \in \mathcal{C}_{\nu^+}$, if $ x \leq y$, then 
$x+z \leq y+z$ and $-y \leq -x$.
\end{lem}

\begin{proof}
Since $\nu^+(x-y)=0$, we have
$$
\nu^+((x+z)-(y+z))=\nu^+(x-y)=0
$$
and
$$
\nu^+((-y)-(-x))= \nu^+ (x-y)=0.
$$
\end{proof}

\begin{lem}
\label{lem partial3}
For coprime integers $p,q>0$, $k \in \mathbb{Z}_{\geq 0}$ and $0 \leq i \leq p-1$,
all of $\tau$, $-\Upsilon$, $V_k$, $\nu^+$
and $-d(S^3_{p,q}(\cdot),i)$ preserve the partial order.
\end{lem}

\begin{proof}
It is proved in \cite{hom-wu} and \cite[Proposition 4.7]{OSS} that
$$
\tau(x) \leq \nu^+(x)
$$
and
$$
-\Upsilon_x (t) \leq (1-|1-t|) \nu^+(x)
$$
for any $x \in \mathcal{C}_{\nu^+}$ and $t \in [0,2]$. 
Since $\tau$ and $\Upsilon$ are group homomorphisms,
if $x \leq y$, then we have
$$
\tau(x) - \tau(y) = \tau(x-y) \leq \nu^+(x-y) = 0 
$$
and
$$
(-\Upsilon_{x}(t)) - (-\Upsilon_y (t)) =-\Upsilon_{x-y} (t) \leq (1-|1-t|) \nu^+(x-y) =0. 
$$
These imply that $\tau(x) \leq \tau(y)$ and $-\Upsilon_x \leq -\Upsilon_y$.
Furthermore, Theorem \ref{thm BCG} implies that if $x \leq y$, then 
$$
\nu^+(x) = \nu^+(y+ (x-y)) \leq \nu^+(y) + \nu^+(x-y) = \nu^+(y).
$$

To consider $V_k$, we use the following proposition.

\begin{prop}[\text{\cite[Proposition 6.1]{bodnar-celoria-golla}}]
\label{prop BCG}
For any two elements $x,y \in \mathcal{C}_{\nu^+}$ and any $m, n \in \mathbb{Z}_{\geq 0}$, 
we have
$$
V_{m+n}(x+y) \leq V_m(x) + V_n(y).
$$
\end{prop}
This proposition implies that if  $x \leq y$, then for any $k \in \mathbb{Z}_{\geq 0}$,
we have
$$
V_{k}(x) = V_{k+0}(y + (x-y)) \leq V_{k}(y) + V_0(x-y)=V_{k}(y).
$$

Finally we consider $-d(S^3_{p,q}(\cdot),i)$.
By Proposition \ref{prop2.2}, we have
$$
-d(S^3_{p/q}(x),i) = -d(S^3_{p/q}(0),i) + 
2V_{\min\{\lfloor \frac{i}{q} \rfloor, \lfloor \frac{p+q -1-i}{q} \rfloor \}}(x).
$$
Hence if  $x \leq y$, then we have
$$
\begin{array}{lll}
-d(S^3_{p/q}(x),i) &= & -d(S^3_{p/q}(0),i) + 
2V_{\min\{\lfloor \frac{i}{q} \rfloor, \lfloor \frac{p+q -1-i}{q} \rfloor \}}(x)\\
\ & \leq &
 -d(S^3_{p/q}(0),i) + 
2V_{\min\{\lfloor \frac{i}{q} \rfloor, \lfloor \frac{p+q -1-i}{q} \rfloor \}}(y)\\
\ & = & -d(S^3_{p/q}(y),i).
\end{array}
$$
This completes the proof.
\end{proof}

\def\proofname{Proof of Proposition \ref{prop partial}}

\begin{proof}
Proposition \ref{prop partial} immediately follows from the above three lemmas. 
\end{proof}

We next prove Theorem \ref{thm partial}.

\def\proofname{Proof of Theorem \ref{thm partial}}

\begin{proof}
First, suppose that $n=0$ or 1.
Then Theorem \ref{thm1} shows $n^+([J]_{\nu^+} - [K]_{\nu^+})= 0$,
and hence $[J]_{\nu^+} \leq [K]_{\nu^+}$.

Next, suppose that $n \geq 3$.
Then Theorem \ref{thm1} shows
$$
n^+([J]_{\nu^+} - [K]_{\nu^+}) \geq \frac{(n-1)(n-2)}{2}>0,
$$
and hence $[J]_{\nu^+} \nleq [K]_{\nu^+}$.

Finally, suppose that $n\geq 3$ and 
the geometric intersection number between $K$ and $D$ is equal to $n$.
Take a small tubular neighborhood of $D$ (denoted by $\nu(D)$),
and think of the intersection $K \cap \nu(D)$ as a trivial braid with index $n$.
Then $J$ is obtained by replacing $K \cap \nu(D)$ with the pure braid $\Delta_n^2$,
where
$$
\Delta_n = (\sigma_1)(\sigma_2 \sigma_1) \cdots (\sigma_n \cdots \sigma_2 \sigma_1)
$$
(see  \cite[Section 10.5]{cromwell}).
In particular, the pure braid $\Delta_n^2$ has only positive crossings (Figure \ref{pos}), and hence it can be 
deformed into the trivial braid only by changing positive crossing to negative crossing 
(Figure \ref{neg}). 
Such a crossing change is realized by a positive full-twist with 0-linking, as shown in
Figure~\ref{change}. 
This implies that $J$ is deformed into $K$ only by positive full-twists 
with 0-linking. Hence, applying the above argument in the case of $n=0$ repeatedly (and transitivity of the partial order),
we have $[K]_{\nu^+} \leq [J]_{\nu^+}$. On the other hand, since $n \geq 3$, we have 
$[J]_{\nu^+} \nleq [K]_{\nu^+}$. Therefore $[J]_{\nu^+} \neq [K]_{\nu^+}$.
\end{proof}

\begin{figure}[htbp]
\begin{minipage}[]{0.4\hsize}
\hspace{-1mm}
\includegraphics[scale = 0.7]{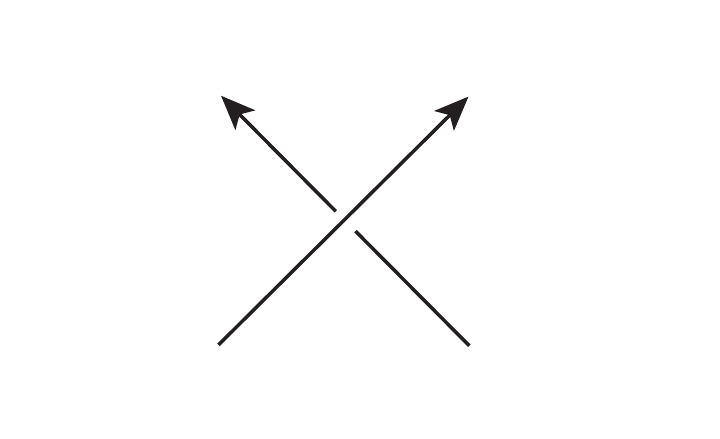}
\vspace{-10mm}
\caption{\label{pos}}
 \end{minipage}
 \begin{minipage}[]{0.4\hsize}
\hspace{-4mm}
\includegraphics[scale = 0.7]{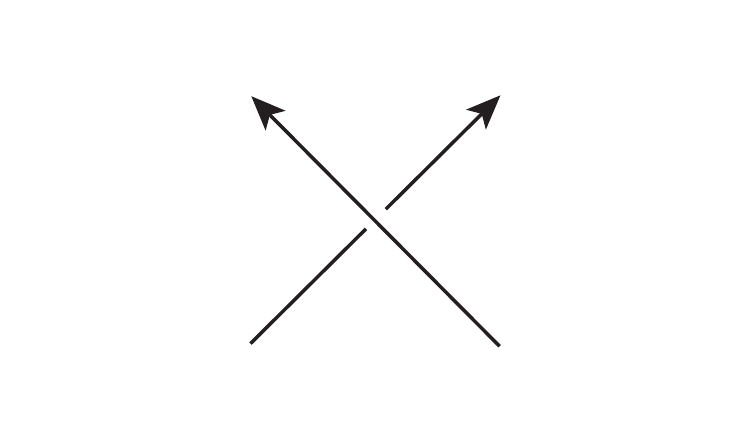}
\vspace{-10mm}
\caption{\label{neg}}
 \end{minipage}
\end{figure}
\vspace{-5mm}

\begin{figure}[htbp]
\hspace{-4mm}
\includegraphics[scale=0.73]{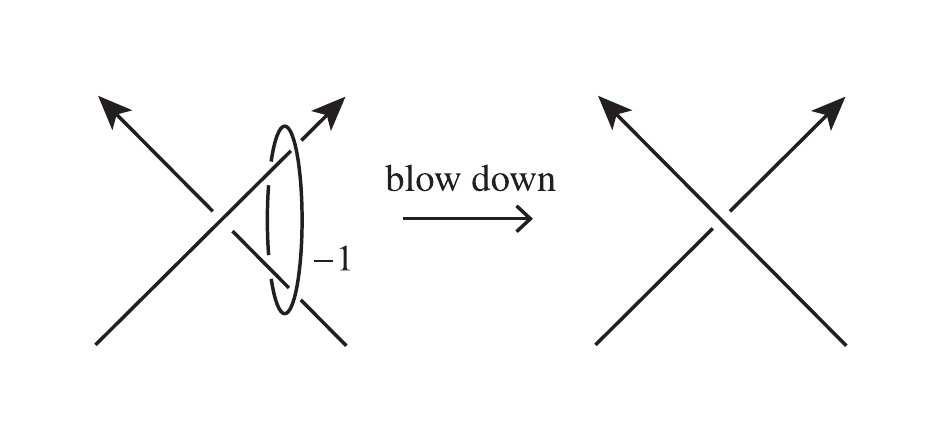}
\vspace{-6mm}
\caption{\label{change}}
\end{figure}

\section{Results on satellite knots}

In this section, we prove Proposition \ref{prop satellite} and 
Corollary \ref{cor satellite}.

\def\proofname{Proof of Proposition \ref{prop satellite}}

\begin{proof}
Let $P$ be a pattern in the standard solid torus $V \subset S^3$.
Note that for any integer $n$, if $P^{n}$ is a pattern obtained from $P$
by performing (-1/n)-surgery along the boundary of the meridian disk of $V$,
then $P^n(K,0)$ is isotopic to $P(K,n)$ for any knot $K$.
Hence we only need to prove that $P_0: [K]_{\nu^+} \mapsto [P(K,0)]_{\nu^+}$ 
is a well-defined order-preserving map for any pattern $P$.
In this proof, we denote $P(K,0)$ simply by $P(K)$.

Let $-P \subset V$ denote the orientation reversed mirror of $P$.
Note that $-(P(K)) = (-P) (-K).$ In order to prove Proposition \ref{prop satellite},
it suffices to prove that for any two knots $K$ and $J$ 
satisfying $V_0(K\#(-J))=0$,  the equality
$$
V_0\big(P(K)\#\big(\text{$-$}P(J) \big) \big)=0
$$
holds. 

By handle calculus, the knot $P(K)\#\big(\text{$-$}P(J)\big)=P(K)\#\big((-P)(-J)\big)$ is 
described as shown in Figure \ref{satellite1}.
Let $C$ denote the (unoriented) center line of $V$ and $\natural$ boundary connected sum.
We start from deforming $P\#(-P)$ into parallel copies of $C\#C$ and in $V\natural V$.

\begin{figure}[htbp]
\hspace{-2mm}
\includegraphics[scale=0.9]{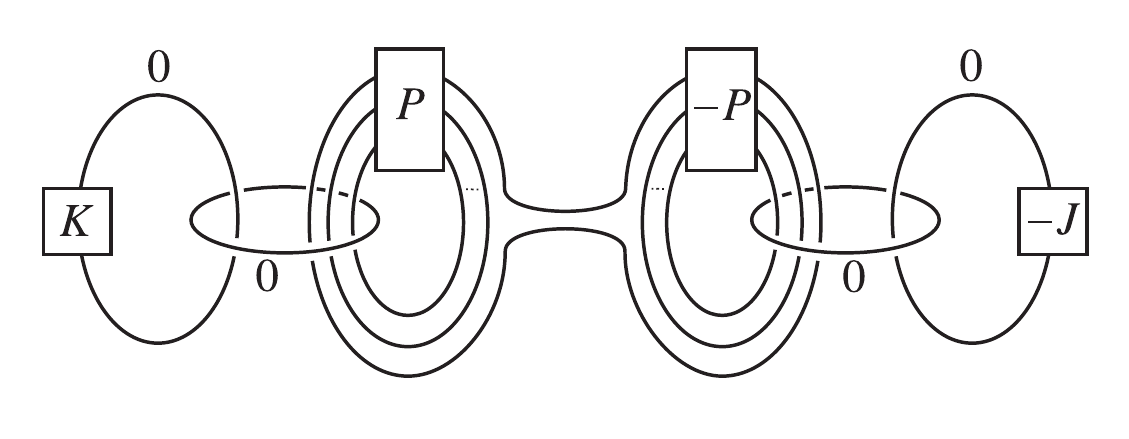}
\vspace{-6mm}
\caption{\label{satellite1}}
\end{figure}

\begin{claim}
For any pattern $P$, there exists an oriented compact surface $S$ properly embedded 
in $(V \natural V) \times [0,1]$ such that 
\begin{enumerate}
\item
$S$ has genus zero,
\item
$S \cap \big( (V \natural V) \times \{0\} \big) = P \# (-P)$, and
\item 
$S \cap \big( (V \natural V) \times \{1\} \big)$ is isotopic to
finitely many parallel copies of $C\#C$ with a certain orientation.
\end{enumerate}  
\end{claim}

\def\proofname{Proof}

\begin{proof}
Identifying $S^1$ with $\Bbb{R}/ \Bbb{Z}$ and
$V$ with $S^1 \times [-1,1] \times [1/2,1]$ respectively.
After an isotopy of $P$,
we can regard $P$ as an proper embedding $P:S^1 \to V$ such that
for small $\varepsilon >0$,
\begin{enumerate}
\item
the projection of $P(S^1)$ on $S^1 \times [-1,1] \times \{ 1/2 \}$ is a regular projection, 
\item
$P(S^1) \cap \big( [-\varepsilon, \varepsilon] \times [-1,1] \times [1/2,1] \big)
= [-\varepsilon, \varepsilon] \times \{ (s_i, t_i ) \}_{i=1}^{n}$ for some $n \in \Bbb{Z}_{>0}$,
\item
$P(t) = (t, 0, 3/4)$ for any $t \in [1/4,3/4]$, and
\item 
$P^{-1}\big( [1/4,3/4] \times \{0\} \times [1/2,3/4) \big)
= \emptyset$.
\end{enumerate}  
Roughly speaking, the above conditions mean
\begin{enumerate}
\item
$P(S^1) \subset S^1 \times [-1,1] \times \{ 1/2 \}$ can be seen as a knot diagram, 
\item
$P(S^1) \cap \big( [-\varepsilon, \varepsilon] \times [-1,1] \times [1/2,1] \big)
$ is a trivial braid,
\item
$P$ contains a half of the center line $[1/4,3/4] \times \{0\} \times \{3/4\}$, and
\item 
$[1/4,3/4] \times \{0\} \times \{3/4\}$ does not contain the over pass for
any crossing on the diagram derived from the condition (1)
\end{enumerate}  
respectively.
Next, for two copies $V_i$ of $V$ ($i=1,2$), we consider a diffeomorphism 
$$f:V_1 \natural V_2 \to \big( S^1 \times [-1,1] \times [-1,1] \big)
\setminus \big( (-\varepsilon/2, \varepsilon/2) \times [-1,1] \times (-1/2, 1/2) \big)$$
so that 
$$
f|_{V_1}(r,s,t) = (r,s, t)
$$
and
$$
f|_{V_2}(r,s,t) = (r,s,-t )
$$
respectively. In particular,
$V_1$ and $V_2$ are identified with
$$
S^1 \times [-1,1] \times [1/2, 1]
$$
and
$$
S^1 \times [-1,1] \times [-1,-1/2]
$$
respectively.
Here, boundary connected sum $V_1 \natural V_2$ is thought of 
as a disjoint union $V_1 \amalg V_2$
with the 1-handle 
$$h^1:=\big(S^1 \setminus (-\varepsilon/2, \varepsilon/2) \big)\times [-1,1] \times [-1/2, 1/2].$$
(See Figure \ref{V}. Note that $V_1\natural V_2$ is the complement of the yellow region.)

\begin{figure}[htbp]
\hspace{-2mm}
\includegraphics[scale=0.8]{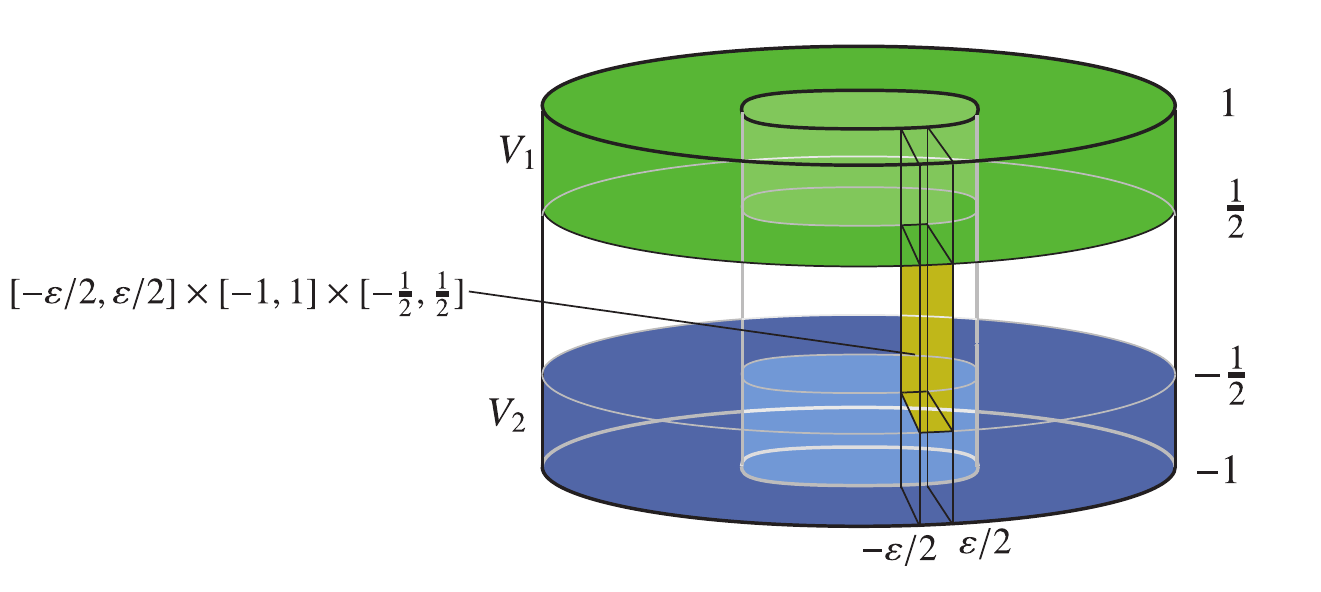}
\vspace{-6mm}
\caption{\label{V}}
\end{figure}

Then, the embedded circle $P\#(-P): S^1 \to V_1 \natural V_2$ is described as 
$$
\big( P\#(-P) \big) (t) = 
\left\{
\begin{array}{lll}
(f|_{V_1}) \circ P\left( -\frac{1}{4} + 2(t +\frac{1}{4}) \right)&\ &
\left(-\frac{1}{4} \leq t \leq 0\right)\\
\ &\ &\ \\
\left(\frac{1}{4}, 0,  \frac{3}{4} - 6t \right)&& \left(0 \leq t \leq \frac{1}{4} \right)\\
\ &\ &\ \\
(f|_{V_2}) \circ P\left( \frac{1}{4} -2(t - \frac{1}{4})  \right)&\ &
\left( \frac{1}{4} \leq t \leq \frac{1}{2} \right)\\
\ &\ &\ \\
\left(-\frac{1}{4}, 0,  -\frac{3}{4} + 6(t - \frac{1}{2}) \right)&& 
\left(\frac{1}{2} \leq t \leq \frac{3}{4} \right)\\
\end{array}
\right. 
.
$$ 
(Note that $(f|_{V_1}) \circ P(S^1)$ and $(f|_{V_2}) \circ P (-S^1)$
are connected by the band $[1/4,3/4] \times [-3/4,3/4]$.  )
Now, we can see that $P\#(-P)$ bounds a ribbon disk in 
$S^1 \times [-1,1] \times [-1,1]$,
which is defined so that 
$$
p_i \circ R(s,t) = p_i \circ \big(P\# (-P) \big) \Big(-\frac{1}{4}(1-t) \Big) 
\ \Big( (s,t) \in [0,1] \times [0,1] \Big)
$$
for $i=1,2$, and
$$
p_3 \circ R(s,t) = p_3 \circ \big(P\# (-P) \big) \Big(-\frac{1}{4}(1-t) \Big) \cdot (2s-1)
\ \Big( (s,t) \in [0,1] \times [0,1] \Big),
$$
where $p_i$ denotes the $i$-th projection of $S^1 \times [-1,1] \times [-1,1]$.
Indeed, we can verify concretely that the boundary of $R$ is equal to
$P\#(-P)$. Moreover, 
any singularity of $R$ is contained in  $\{r\}\times\{s\} \times (-1,1)$,
where $(r,s)$ is the coordinate of a double point on the regular projection $(p_1 \times p_2) \circ P(S^1)$.
Let $1/2 <t_1<t_2<1$ be the 3rd coodinate of points in 
$\big((p_1 \times p_2)^{-1}(r,s) \big) \cap P(S^1)$.  
Then the singularity of $R$ in $\{r\}\times\{s\} \times (-1,1)$ is
equal to $\{r\}\times\{s\} \times [-t_1,t_1]$, which is contained in 
$\{r\}\times\{s\} \times [-t_2,t_2] \subset R$.
Therefore, any singularity of $R$ is ribbon.

Let $R' :=R \cap V_1\natural V_2$.
Then $\partial R$ consists of $P\# (-P)$ and $n$ parallel copies of
$\partial([-\varepsilon/2,\varepsilon/2]\times \{0\} \times [-1/2, 1/2])$
(with a certain orientation).
It is not hard to see that 
$\partial([-\varepsilon/2,\varepsilon/2]\times \{0\} \times [-1/2, 1/2])$
is isotopic to the connected sum of the longitude of $V$ in $\partial(V \natural V)$,
and hence it is isotopic to $C\#C$ in $V \natural V$.
Since $R'$ is a ribbon surface, we can construct a cobordism
$S$ in $(V_1\natural V_2 )\times [0,1]$ from $P\#(-P)$ to $n$ times parallel
copies of $C \# C$, which is homeomorphic to $R'$.
\end{proof}

\begin{figure}[htbp]
\hspace{-2mm}
\includegraphics[scale=0.8]{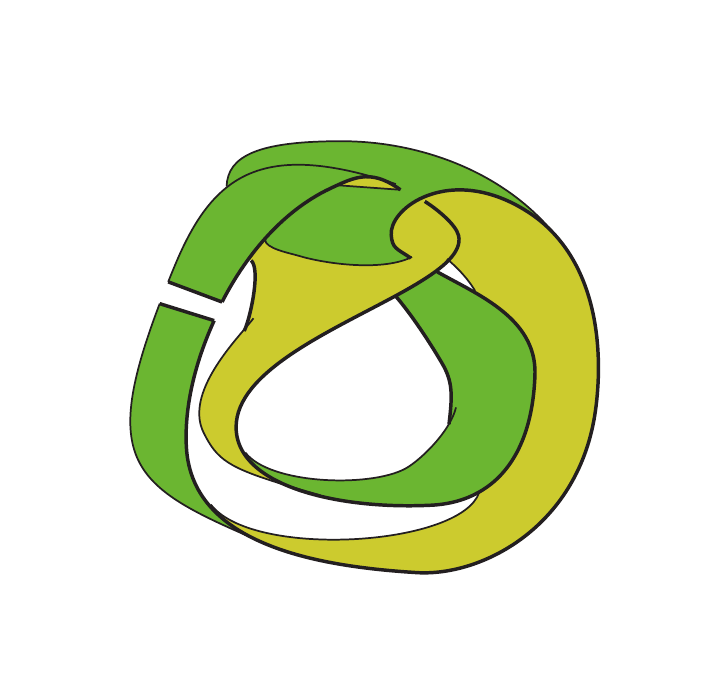}
\vspace{-6mm}
\caption{\label{R} A ribbon disk $R$ for the Whitehead double}
\end{figure}

For a knot $K$ in $S^3$,
let $W_0(K)$ denote the 4-manifold obtained by attaching a 0-framed 2-handle to
$S^3 \times [0,1]$ along $K \subset S^3 \times \{1\}$.

\begin{claim}
The knot $-\big(P(K) \# \big(\text{$-$}P(J) \big)\big) \subset -(S^3 \times \{ 0\})$ bounds
a disk in the 4-manifold $W_0(K\#(-J))$.
\end{claim}

\def\proofname{Proof}

\begin{proof}
We can extend the embedding of $P\#(-P)$ shown in Figure \ref{satellite1} to $V \natural V$,
as shown in Figure \ref{satellite2}.
By using the cobordism in Claim 1, we have
a cobordism in $S^3 \times [0,1]$ with genus zero 
which connects $P(K) \# \big(\text{$-$}P(J)\big)$ to the link shown in Figure \ref{satellite3}.
Furthermore, it follows from elementary handle calculus that the link in Figure \ref{satellite3}
is isotopic to the $(n,0)$-cable of $K\# (-J)$ for some positive integer $n$.
By attaching a 0-framed 2-handle along $K\#(-J) \subset S^3 \times \{ 1 \}$
and capping off the cable $(K\#(-J))_{n,0}$ with $n$ parallel copies of the core of the 2-handle,
we obtain a disk in $W_0(K\#(-J))$ with boundary $-\big(P(K) \# \big(\text{$-$}P (J) \big)\big) \subset -(S^3 \times \{ 0\})$.
\end{proof}

\begin{figure}[htbp]
\hspace{-2mm}
\includegraphics[scale=0.9]{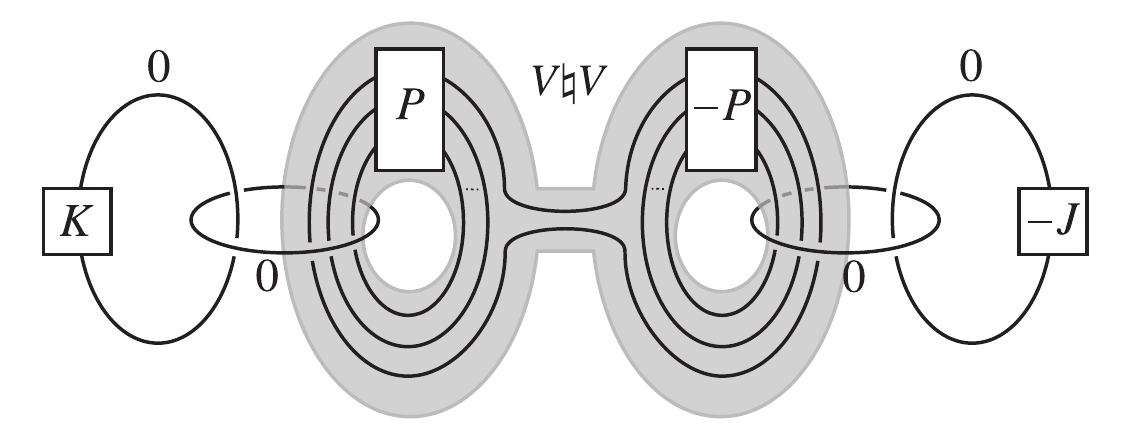}
\vspace{-6mm}
\caption{\label{satellite2}}
\hspace{-2mm}
\includegraphics[scale=0.9]{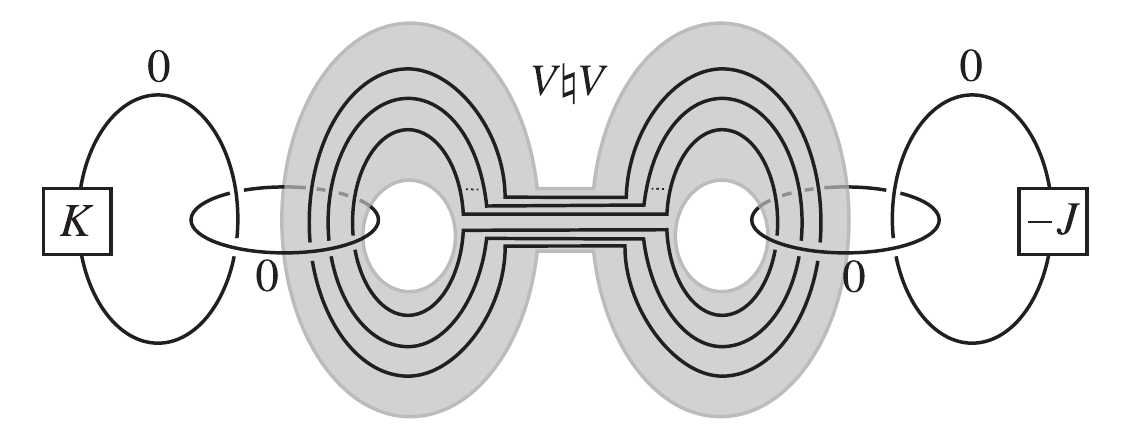}
\vspace{-6mm}
\caption{\label{satellite3}}
\end{figure}

Attach a $(-1)$-framed 2-handle to $W_0(K)$ along $-\big(P(K) \# \big((-P (-J) \big)\big) \subset -(S^3 \times \{ 0\})$, and cap off the disk in Claim 2 with the core.
Then the self-intersection of the resulting sphere is $-1$, and hence we can blow down the sphere. Let $W$ denote the resulting cobordism. Then we can see that 
$$\partial W = -S^3_1\big(P(K) \# \big(\text{$-$}P(J) \big)\big) \amalg S^3_0(K \# (-J)).$$
Taking a properly embedded arc in $W$ from 
$-S^3_1\big(P(K) \# \big(\text{$-$}P (J)\big)\big)$ to $S^3_0(K \# (-J))$ and
removing the small open tubular neighborhood of the arc,
we obtain a compact oriented 4-manifold $X$ with boundary 
$-S^3_1\big(P(K) \# \big(\text{$-$}P(J) \big)\big) \# S^3_0(K \# (-J))$. Moreover,
it follows from elementary homology theory that 
$$H_*(X;\Bbb{Z}) \cong H_*(S^2 \times D^2; \Bbb{Z}).$$
Now we apply the following theorem to $-X$.

\begin{thm}[\text{\cite[Corollary 9.13]{ozsvath-szabo}}]
\label{s2d2}
Suppose that $Y$ is a 3-manifold with $H_1(Y;\Bbb{Z}) \cong \Bbb{Z}$.
If $Y$ bounds an integer homology $S^2 \times D^2$, then
$d_{-1/2}(Y)\geq -1/2$.
\end{thm}

By Theorem \ref{s2d2} and Corollary \ref{cor2.3}, we have
$$
\begin{array}{lll}
-1/2 & \leq 
& d_{-1/2}(-S^3_0(K \# (-J))) + d\big(S^3_1\big(P(K) \# \big(\text{$-$}P(J) \big)\big) \big)\\
\ &=& \  -1/2 + d(S^3_{-1}(J \# (-K)))  -2V_0 \big(P(K) \# \big(\text{$-$}P(J)\big)\big)\\
\ &=& \  -1/2 + 2V_0(K \# (-J)))  -2V_0\big(P(K) \# \big(\text{$-$}P(J)\big)\big).
\end{array}
$$
Since $V_0(K \# (-J))=0$, we have
$$
0 \leq V_0\big(P(K) \# \big(\text{$-$}P(J)\big)\big) \leq V_0(K \# (-J))=0
$$ 
and hence 
$
V_0\big(P(K) \# \big(\text{$-$}P(J)\big)\big)=0$.
\end{proof}

Finally, we prove Corollary \ref{cor satellite}.

\def\proofname{Proof of Corollary \ref{cor satellite}}

\begin{proof}
If $D$ denotes the image of the meridian disk of $V$ by the embedding $e_n$,
then it is easy to see that $D$ intersects $P(K,n)$ in its interior, $\lk(P(K,n),D)=w(P)$
and $P(K,n)$ is deformed into $P(K,n+1)$ by a positive full-twist along $D$. 

We first suppose that $w(P) = 0$ or 1.
Let $x \in \mathcal{C}_{\nu^+}$, the symbol $K$ denote a representative of $x$ and $m,n$ integers with $m<n$.
Then, by applying Theorem \ref{thm partial} to the pair $(P(K,m),D)$ repeatedly,
we have 
$$
P_m(x) = [P(K,m)]_{\nu^+} \geq [P(K,m+1)]_{\nu^+} \geq \ldots 
[P(K,n)]_{\nu^+}=P_{n}(x).
$$

Next we suppose that $w(P) \geq 3$ and the geometric intersection number between the pattern $P$ and the meridian disk of $V$ is equal to $w(P)$. Then the embedding $e_n$ preserves the number of intersection points, and hence the geometric intersection number between $D$ and $P(K,n)$ is also equal to $n$. By applying Theorem \ref{thm partial} to
the pair $(P(K,n),D)$ repeatedly, we have
$$
P_m(x) = [P(K,m)]_{\nu^+} < [P(K,m+1)]_{\nu^+} <
[P(K,n)]_{\nu^+}=P_{n}(x).
$$
This completes the proof.
\end{proof}

\end{document}